%% file: main.tex
\documentclass[a4paper,10pt]{article}

\input{settings}

\title{Deciding reducibility of mapping classes is in $\NP$}
\author{Mark C. Bell\\
University of Illinois\\
\url{mcbell@illinois.edu}}

\begin{document}

\maketitle

\begin{abstract}
For a fixed marked surface $S$, we show that the problem of deciding whether or not a mapping class is reducible lies in $\NP$. As usual this immediately gives an exponential time algorithm to decide whether or not a mapping class is reducible.

To do this we use an (ideal) triangulation to obtain a coordinate system on the set of multicurves on $S$. The result then follows from the fact that the action of the mapping class group of $S$ is piecewise-linear with respect to such a coordinate system and so we are able so show that: if a mapping class $h$ fixes a multicurve then it fixes one whose size is at most exponential in the word length of $h$.

We go on to show how to repeat this construction on invariant subsurfaces. This allows us to show that a similar bound holds for the size of the canonical curve system of a mapping class and so give an alternate, elementary proof of a result of Koberda and Mangahas.
\end{abstract}


\ccode{57M99}

\input{introduction}
\input{triangulations}

\input{classification_problem}
\input{subsurfaces}

\bibliographystyle{plain}
\bibliography{bibliography}

\end{document}

%% file: settings.tex
\usepackage{amssymb, amsmath, amscd, amsfonts, amsthm}
\usepackage{url}
\usepackage{colonequals}
\newcommand{\defeq}{\colonequals}

\usepackage{tikz}
\usetikzlibrary{calc, arrows, decorations.markings}

\usepackage{graphicx}

\theoremstyle{definition}
\newtheorem{definition}[equation]{Definition}
\newtheorem{remark}[equation]{Remark}
\newtheorem{problem}[equation]{Problem}
\newtheorem*{keywords}{keywords}

\theoremstyle{plain}
\numberwithin{equation}{section}
\newtheorem{theorem}[equation]{Theorem}
\newtheorem{corollary}[equation]{Corollary}
\newtheorem{proposition}[equation]{Proposition}
\newtheorem{lemma}[equation]{Lemma}

\newcommand{\fakeenv}{} %%% \fakeenv prints the emptystring

\newenvironment{restate}[2]  %%% restate takes two arguments 
{ 
 \renewcommand{\fakeenv}{#2} %%% Now \fakeenv prints #2
 \theoremstyle{plain} 
 \newtheorem*{\fakeenv}{#1~\ref{#2}} %%% so now #2 is the name of a
 \begin{\fakeenv}
}
{
 \end{\fakeenv}
}

\newtheoremstyle{dotless}{}{}{}{}{\bfseries}{}{ }{}
\theoremstyle{dotless}
\newtheorem*{ccode}{\textbf{Mathematics Subject Classification (2010):}}

\newcommand{\inlineand}{\quad \textrm{and} \quad}
\newcommand{\inlineQED}{\pushQED{\qed} \qedhere \popQED}

\DeclareMathOperator{\Id}{Id}
\DeclareMathOperator{\intersection}{\iota}
\newcommand{\NP}{\textbf{NP}}
\DeclareMathOperator{\sgn}{sign}
\DeclareMathOperator{\Mod}{Mod^+} % Modular group
\DeclareMathOperator{\nummarkedpoints}{n}
\let\genus\relax
\DeclareMathOperator{\genus}{g}  % This is how the genus should be denoted.

\newcommand{\NN}{\mathbb{N}}
\newcommand{\calC}{\mathcal{C}}
\newcommand{\calT}{\mathcal{T}}

\newcommand{\superscript}[1]{\ensuremath{^{\textrm{#1}}}}
\newcommand{\nth}{\superscript{th}}

\newcommand{\from}{\colon}

\newcommand{\canonical}{\sigma}

%% file: introduction.tex
\section{Introduction}

Fix $S$ to be a (possibly disconnected) marked surface in which each component contains at least one marked point and no component is a once or twice marked sphere.

Let $\calC(S)$ denote the set of essential, simple, closed \emph{multicurves} on $S$. This is strictly larger than the set of simplices of the curve complex of $S$ \cite{MasurMinskyII}; it includes multicurves in which some of the components are parallel.

Let $\Mod(S)$ denote the \emph{mapping class group} of $S$, relative to the set of marked points. We fix $X$ to be a finite generating set of $\Mod(S)$ and let $X^*$ denote the set of all words that can be made using the elements of $X$ as letters. We identify a word $h = h_1 \cdots h_k \in X^*$ with the mapping class 
\[ h_k \circ \cdots \circ h_1 \]
and denote its \emph{length} by $\ell(h) \defeq k$. 

\begin{definition}
A mapping class $h \in \Mod(S)$ is \emph{reducible} if there is an \emph{$h$--invariant} multicurve, that is, a multicurve $\gamma \in \calC(S)$ such that $h(\gamma) = \gamma$. A word is \emph{reducible} if its corresponding mapping class is.
\end{definition}

\begin{problem}[The Reducibility Problem]
Given a word $h \in X^*$, decide whether or not $h$ is reducible.
\end{problem}

We begin in Section~\ref{sec:triangulation} by describing how to obtain a coordinate system on $\calC(S)$ from an (ideal) triangulation. This allows us to efficiently represent multicurves by a vector of integers. In Section~\ref{sub:graph_triangulations} we also use these triangulations to give a combinatorial way of representing mapping classes and in Section~\ref{sub:encoding} we show how this can be used to efficiently compute the image of a multicurve under a mapping class.

In Section~\ref{sec:classification_problem} we use the fact that the action of $\Mod(S)$ on $\calC(S)$ is piecewise-linear with respect to this coordinate system to show that: if there is an $h$--invariant multicurve then there is one that is small with respect to our coordinate system (Corollary~\ref{cor:invariant_bound}). Such an invariant multicurve acts as a certificate of reduciblity and is sufficiently small that it can be verified that it is invariant in polynomial time. Thus we deduce that:

\begin{restate}{Corollary}{cor:reducible_NP}
The reducibility problem is in $\NP$.
\end{restate}

Finally, in Section~\ref{sec:subsurfaces}, we describe how to repeat this construction on invariant subsurfaces. Using these structures again, we obtain similar bounds and so are able to also give an elementary proof of a result of Koberda and Mangahas \cite[Theorem~1]{KoberdaMangahas}:

\begin{restate}{Corollary}{cor:canonical_bound}
Fix $\calT$, a triangulation of $S$. For each word $h \in X^*$, the edge vector $\calT(\canonical(h))$ of the canonical curve system of $h$ is $O(\ell(h))$--bounded.
\end{restate}

Unlike the proof given by Koberda and Mangahas, the proof of this bound does not rely on any knowledge of the proof of solvability of the conjugacy problem for $\Mod(S)$, constants related to the curve complex or the finite index classifying covers of $S$.

% Finally, we remark that the techniques that we will describe can also be directly applied to the \emph{extended mapping class group} $\EMod(S)$ which includes orientation-reversing mapping classes \cite[Page 219]{FarbMargalit}.

\subsection{Notation}
\label{sub:notation}

We begin by setting some general notation and noting that all vectors and matrices will have integer entries throughout. Let:
\begin{itemize}
\item $\NN$ denote the set of natural numbers including zero,
\item $|S|$ denote the number of components of $S$,
\item $\genus(S)$ denote the genus of $S$, which we define to be the sum of the genuses of the components when $S$ is disconnected,
\item $\nummarkedpoints(S)$ denote the number of marked points on $S$,
\item $v[i]$ denote the $i\nth{}$ entry of a vector $v$,
\item $v \geq 0$ denote that the vector $v$ is \emph{non-negative}, that is, each entry of $v$ is non-negative,
\item $v \geq_2 0$ denote that $v \geq 0$ and that each entry of $v$ is even, and
\item $\left(\begin{array}{c}A\\ \hline B \end{array}\right)$ denote the join of matrices $A$ and $B$, obtained by stacking their rows.
\end{itemize}

\subsection{Model of computation}
\label{sub:model}

To simplify our analysis we will assume that each variable can hold an arbitary integer and there is is no cost associated to variable access. However, in the problems that we will tackle not only will matrices grow in size with the problem but so will the entries involved. This growth is almost always exponential and so we must take extra care in our analysis.

\begin{definition}
\label{def:bounded}
An integer, vector or matrix is \emph{$k$--bounded} if the log of the absolute value of each entry is at most $k$. That is, if each number involved can be represented by at most $k$ bits.
\end{definition}

As part of our model of computation, we will assume that if $x$ and $y$ are $k$--bounded and $k'$--bounded integers respectively where $k \geq k'$ then:
\begin{itemize}
\item $\sgn(x)$ can be computed in $O(1)$ operations,
\item $x \pm y$ is $(k + 1)$--bounded and can be computed in $O(k)$ operations, and
\item $xy$ is $(k+k')$--bounded and can be computed in $O(k k')$ operations.
\end{itemize}

%% file: triangulations.tex
\section{Triangulations}
\label{sec:triangulation}

\begin{definition}
An \emph{(ideal) triangulation} $\calT$ of $S$ is the isotopy class of a finite, ordered collection of arcs on $S$ which connect between the marked points, have pairwise disjoint interiors and are such that the metric completion of each component of $S - \calT$ is an unmarked triangle.
\end{definition}

When working with a triangulation, we refer to the marked points as \emph{vertices}, the arcs as \emph{edges} and the metric completion of each component of $S - \calT$ as \emph{faces}. We let 
\[ \zeta = \zeta(S) \defeq 6 \genus(S) + 3 \nummarkedpoints(S)- 6 |S| \]
denote the \emph{complexity} of $S$. This is the number of edges of any triangulation of $S$.

The fact that the edges of a triangulation are ordered will be crucial. Changing the ordering of the edges of a triangulation does not produce an equivalent triangulation.

\begin{definition}
Let $\calT$ be a triangulation of $S$ with edges $e_1, \ldots,e_\zeta$ (in order). The \emph{edge vector} of a multicurve $\gamma \in \calC(S)$ with respect to $\calT$ is the vector
\[ \calT(\gamma) \defeq
\left(
\begin{array}{c}
\intersection(\gamma, e_1) \\
\vdots \\
\intersection(\gamma, e_\zeta) \\
\end{array}
\right) \in \NN^\zeta
\]
where $\intersection(x, y)$ is the \emph{geometric intersection number} of $x$ and $y$.
\end{definition}

Although for each triangulation $\calT$ the map $\calT(\cdot) \from \calC(S) \to \NN^\zeta$ is injective it is not surjective. In fact a vector $v \in \NN^\zeta$ corresponds to a multicurve if and only if:
\begin{itemize}
\item for each face of $\calT$ with edges $e_i$, $e_j$ and $e_k$ we have that 
$v[i] + v[j] - v[k] \in 2 \NN$,
and 
\item for each vertex $v$ of $\calT$ there is a face with edges $e_i$, $e_j$ and $e_k$ such that $e_i \cap e_j = v$ and $v[i] + v[j] - v[k] = 0$.
\end{itemize}
We express this requirement as a linear programming problem:

\begin{lemma}
\label{lem:is_multicurve_LP}
For each triangulation $\calT$, there are $O(1)$--bounded $\zeta \times 3\zeta$ matrices $F_1, \ldots, F_k$ such that $v \in \NN^\zeta$ is in the image of $\calT(\cdot)$ if and only if $v \neq 0$ and
\[ F_i \cdot v \geq_2 0 \]
for some $i$. \qed
\end{lemma}

This lemma allows us to test whether a $k$--bounded vector is in the image of $\calT(\cdot)$ in $O(k)$ operations.

\subsection{The graph of (ordered) triangulations}
\label{sub:graph_triangulations}

The edge vector of a multicurve depends heavily on the choice of triangulation. We consider two elementary ways of altering a triangulation, both of which change the edge vector of a multicurve predictably. 

Firstly, we may use a permutation to reorder the edges of a triangulation.

\begin{lemma}
\label{lem:permutation_action}
Let $\calT$ and $\calT'$ be triangulations of $S$ which are equivalent up to reordering. Let $\Sigma$ be the permutation matrix corresponding to the reordering. Then for each multicurve $\gamma \in \calC(S)$,
\[ \calT'(\gamma) = \Sigma \cdot \calT(\gamma). \inlineQED{} \]
\end{lemma}

Secondly, if the interior of an edge $e$ meets two distinct faces of $\calT$ then we may \emph{flip} it to obtain a new triangulation $\calT'$. This is done by replacing $e$ with $e'$, the opposite diagional of the square containing $e$, as shown in Figure~\ref{fig:flip}. 

\begin{figure}[ht]
\centering
\begin{tikzpicture}[scale=2.5,thick]
\coordinate (A) at (-1,0);
\coordinate (B) at (0,0);
\coordinate (C) at (0,1);
\coordinate (D) at (-1,1);

\draw (C) -- node[above] {$a$} (D);
\draw (B) -- node[right] {$d$} (C);
\draw (A) -- node[below] {$c$} (B);
\draw (A) -- node[left]  {$b$} (D);
\draw (A) -- node[above left] {$e$} (C);

\node (t) at (-0.5,-0.25) {$\calT$};

\coordinate (A2) at (1,0);
\coordinate (B2) at (2,0);
\coordinate (C2) at (2,1);
\coordinate (D2) at (1,1);

\draw (C2) -- (D2);
\draw (B2) -- (C2);
\draw (A2) -- (B2);
\draw (A2) -- (D2);
\draw (B2) -- node[above right] {$e'$} (D2);

\node (t2) at (1.5,-0.25) {$\calT'$};

\draw[thick,->] (0.25,0.5) -- node[above] {Flip} (0.75,0.5);
\end{tikzpicture}
\caption{Flipping an edge of a triangulation.}
\label{fig:flip}
\end{figure}
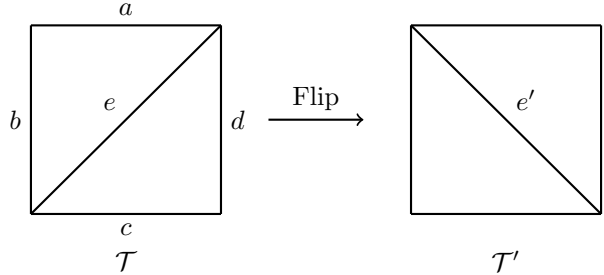

\begin{proposition}[{\cite[Page 30]{Mosher}}]
\label{prop:Pachner_change}
Let $\calT$ and $\calT'$ be triangulations of $S$ and suppose that $\calT'$ is obtained from $\calT$ by flipping the edge $e$. Let $a$, $b$, $c$ and $d$ be the four edges of $\calT$ shown in Figure~\ref{fig:flip}. Then for each multicurve $\gamma \in \calC(S)$,
\[ \intersection(\gamma, e') = \max(\intersection(\gamma, a) + \intersection(\gamma, c), \intersection(\gamma, b) + \intersection(\gamma, d)) - \intersection(\gamma, e). \inlineQED \]
\end{proposition}

We encapsulate these two moves in a simplicial \emph{graph of triangulations} $G = G(S)$. The vertices of $G$ correspond to triangulations of $S$. Two vertices in $G$ are connected by an edge if and only if their triangulations differ by a reordering of their edges or by a single flip.

We assign each edge of this graph length one and denote the \emph{length} of a path $p$ in $G$ with respect to the induced path metric by $\ell(p)$. 

\subsection{Encoding mapping classes}
\label{sub:encoding}

We note that $G$ is connected \cite[Page 190]{Hatcher} and that the natural action of $\Mod(S)$ on $G$ is \emph{geometric}, that is, a properly discontinuous, cocompact action by isometries. Therefore $G$ and $\Mod(S)$ are quasi-isometric \cite[Proposition~8.19]{BridsonHaefliger} and so we can use paths in $G$ to combinatorally represent mapping classes. Specifically, we represent $h \in \Mod(S)$ via a path $p$ in $G$ from $\calT$ to $h(\calT)$. 

One method of obtaining such a path is to first fix paths representing each of the generators in $X$. Then we can construct a path $p$ representing $h = h_1 \cdots h_k \in X^*$ by concatinating together translated copies of the paths representing $h_1, \ldots, h_k$. Using this construction, $\ell(p) \in O(\ell(h))$ 

Now suppose that $p$ is a path from $\calT$ to $\calT'$. If we are given $v = \calT(\gamma)$ then we may use Lemma~\ref{lem:permutation_action} and Proposition~\ref{prop:Pachner_change} to compute $\calT'(\gamma)$.

\begin{lemma}
\label{lem:compute_curve}
Suppose that $p$ is a path from $\calT$ to $\calT'$ and that $\gamma \in \calC(S)$ is a multicurve. If $v \defeq \calT(\gamma)$ is $k$--bounded then $v' \defeq \calT'(\gamma)$ is $(k+\ell(p))$--bounded and can be computed in at most
\[ O\left(k \ell(p) + \ell(p)^2\right) \]
operations.
\end{lemma}

\begin{proof}
When $p$ consists of a single relabelling or flip, Lemma~\ref{lem:permutation_action} and Proposition~\ref{prop:Pachner_change} show that $v'$ is $(k+1)$--bounded and can be computed in at most $O(k)$ operations. The result then hold by induction on $\ell(p)$.
\end{proof}

In the case when $p$ is a path representing $h \in \Mod(S)$, as 
\[ \calT(h(\gamma)) = h^{-1}(\calT)(\gamma), \]
this lemma allows us to compute $\calT(h(\gamma))$ from $\calT(\gamma)$ in $O\left(k \ell(p) + \ell(p)^2\right)$ operations.

\begin{remark}
By Alexander's trick, a mapping class fixes every edge vector if and only if it fixes every $1$--bounded edge vector \cite[Proposition~2.8]{FarbMargalit}. However, if $S$ is connected and not a once-marked torus or four times marked sphere then the only mapping class which fixes every edge vector is the identity map. Thus, by testing whether these vectors are fixed by $h$ we can determine whether $h = \Id$ in at most 
\[ O\left(\ell(h)^2\right) \]
operations. This gives the same asymptotic bound on the word problem for mapping class groups as Mosher obtained by showing that $\Mod(S)$ is automatic \cite[Section~3]{Mosher2}.
\end{remark}

Now observe that, as the functions used in Lemma~\ref{lem:permutation_action} and Proposition~\ref{prop:Pachner_change} are piecewise linear\footnote{Here piecewise linear functions are required to have a finite number of cells, each defined by a system of linear inequalities.}, for each pair of triangulations $\calT$ and $\calT'$ there is also a piecewise linear function $f \from \NN^\zeta \to \NN^\zeta$ such that
\[ \calT'(\gamma) = f(\calT(\gamma)) \]
for each muticurve $\gamma \in \calC(S)$. We express this piecewise-linear function using two collections of matrices $\{A_i\}$ and $\{B_i\}$. Here the matrix $A_i$ describes the linear transformation inside of the $i$\nth{} cell of $f$ while the matrix $B_i$ describes the system of linear inequalities which define the $i$\nth{} cell. Given a path  $p$ from $\calT$ to $\calT'$ we can construct these matrices as follows:
\begin{itemize}
\item If $p$ consists of a single reordering then we define its matrices to be:
\begin{eqnarray*}
A_1 &\defeq& \Sigma \\
B_1 &\defeq& (0 \; \cdots \; 0)
\end{eqnarray*}
where $\Sigma$ is the permutation matrix corresponding to the relabelling.
\item If $p$ consists of a single flip of an edge $e$ of $\calT$, as shown in Figure~\ref{fig:flip}, then we define its matrices to be:
\begin{eqnarray*}
A_1 &\defeq& \Id + E_{ea} + E_{ec} - 2 E_{ee} \\
A_2 &\defeq& \Id + E_{eb} + E_{ed} - 2 E_{ee} \\
B_1 &\defeq& E_a + E_c - E_b - E_d \\
B_2 &\defeq& E_b + E_d - E_c - E_a
\end{eqnarray*}
where $\Id$ is the identity matrix, $E_i$ is the $\zeta \times 1$ matrix with a $1$ at position $(i, 1)$ and $0$ everywhere else and $E_{ij}$ is the $\zeta \times \zeta$ matrix with a $1$ at position $(i, j)$ and $0$ everywhere else.
\item Otherwise we decompose $p$ as $p' \cdot p''$ and inductively define its matrices to be:
\[ A_i \defeq A''_k \cdot A'_j \inlineand B_i \defeq \left(\begin{array}{c}
B'_j \\
\hline
B''_k \cdot A'_j
\end{array}\right) \]
where $\{A'_j\}$ and $\{B'_j\}$ are the matrices of $p'$ and $\{A''_k\}$ and $\{B''_k\}$ are the matrices of $p''$.
\end{itemize}

\begin{lemma}
\label{lem:encoding_LP}
Suppose that $p$ is a path from $\calT$ to $\calT'$. Let $\{A_i\}$ and $\{B_i\}$ be the matrices defined above.
\begin{enumerate}
\item Each $A_i$ and $B_i$ is $\ell(p)$--bounded.
\item Each $B_i$ has $O(\ell(p))$ rows.
\item For each multicurve $\gamma \in \calC(S)$ we have that $B_i \cdot \calT(\gamma) \geq 0$ for some $i$.
\item For each multicurve $\gamma \in \calC(S)$, if $B_i \cdot \calT(\gamma) \geq 0$ then $\calT'(\gamma) = A_i \cdot \calT(\gamma)$. \qed
\end{enumerate}
\end{lemma}

%% file: classification_problem.tex
\section{Determining and certifying reducibility}
\label{sec:classification_problem}

We will express the reducibility problem as an linear programming problem. Small invariant curves will then correspond to small solutions to this linear programming problem. This is closely related to the vertex enumeration problem for unbounded polytopes \cite{BorosElbassionGurvichMakino}. We start with a technical lemma for bounding determinants of matrices.

\begin{lemma}
\label{lem:matrix_bounded}
If $M$ is a $k$--bounded, $n \times n$ matrix then $\det(M)$ is $(k n + n \log(n) / 2)$--bounded.
\end{lemma}

\begin{proof}
This bound follows immediately from Hadamard's inequality \cite[Theorem~14.1.1]{Garling}.
\end{proof}

\begin{proposition}
\label{prop:small_polytope_vectors}
Suppose that $M$ is a $k$--bounded, $m \times n$ matrix. If the polytope
\[ P \defeq \{ v \in \NN^n : M \cdot v \geq 0 \} \]
is non-trivial then it contains a non-trivial $(n k + n \log(n) / 2)$--bounded integral vector.
\end{proposition}

\begin{proof}
Without loss of generality we may assume that the basis vectors $E_i$ are rows of $M$. Let $v_0$ be an \emph{extremal} vector of $P$, that is, $v_0 \in P - \{0\}$ and there are $n - 1$ linearly independent rows $r_1, \ldots, r_{n-1}$ of $M$ such that $r_i \cdot v_0 = 0$. We claim that we can rescale $v_0$ to obtain $v_1 \in P - \{0\}$, a vector in which each entry is a $(n k + n \log(n) / 2)$--bounded integer.

To see this, define $r_0 \defeq (1 \; \cdots \; 1)$ and let $A$ be the matrix with rows $r_0, r_1, \ldots, r_{n-1}$. Then $v_0$ is the unique solution to
\[ A \cdot v = 
||v_0|| \cdot 
\left(\begin{array}{c}
1 \\
0 \\
\vdots \\
0
\end{array}\right) \]
By Cramer's rule, if $A_i$ is the matrix obtained by replacing the $i$\nth{} column of $A$ by $(1\ 0\ \cdots\ 0)^T$ then the $i$\nth{} entry of $v_0$ is given by 
\[ ||v_0|| \cdot \frac{\det(A_i)}{\det(A)}. \]
Hence, by rescaling $v_0$ by $|\det(A)| / ||v_0||$ we obtain a vector $v_1 \in P - \{0\}$ whose $i$\nth{} entry is $|\det(A_i)|$. However $v_1$ is $(n k + n \log(n) / 2)$--bounded as each $|\det(A_i)|$ is by Lemma~\ref{lem:matrix_bounded}.
\end{proof}

\begin{theorem}
\label{thrm:invariant_bound}
Suppose that $h \in \Mod(S)$ is a mapping class and that $p$ is a path from $\calT$ to $h(\calT)$. If $h$ is reducible then there is an $h$--invariant multicurve $\gamma \in \calC(S)$ such that $\calT(\gamma)$ is $O(\ell(p))$--bounded.
\end{theorem}

\begin{proof}
Let $\{ A_i \}$ and $\{ B_i \}$ be the matrices of Lemma~\ref{lem:encoding_LP}. Additionally, let $\{ F_j \}$ be the matrices of Lemma~\ref{lem:is_multicurve_LP}. Then for each $i$ and $j$, let
\[ M(i, j) \defeq
\left(\begin{array}{c}
A_i - \Id \\
\hline
-(A_i - \Id) \\
\hline
B_i \\
\hline
F_j \\
\hline
\Id
\end{array}\right)\]
We begin by claiming that $h$ is reducible if and only if there is a non-trivial solution to $M(i, j) \cdot v \geq 0$ for some $i$ and $j$.

To prove this claim, firstly suppose that $h(\gamma) = \gamma$ and let $v \defeq \calT(\gamma) \neq 0$. Let $i$ be such that $B_i \cdot v \geq 0$ and so $A_i \cdot v = \calT(h(\gamma)) = v$. Hence
\[ (A_i - \Id) \cdot v \geq 0 \inlineand -(A_i - \Id) \cdot v \geq 0. \]
As $\gamma$ is a multicurve there is a $j$ such that $F_j \cdot v \geq 0$. Thus $v$ is a non-trivial vector and $M(i, j) \cdot v \geq 0$.

Conversely, suppose that $v$ is a non-trivial solution to $M(i, j) \cdot v \geq 0$. Without loss of generality we may assume that the entries of $v$ are non-negative and rational as $M(i, j)$ defines a rational polytope. Furthermore, by scaling $v$ by a sufficiently large natural number we may assume that 
\[ F_j \cdot v \geq_2 0. \]
Hence, there is a multicurve $\gamma \in \calC(S)$ such that $\calT(\gamma) = v$. As $B_i \cdot v \geq 0$ and $v$ lies in the kernel of $A_i - \Id$, we have that $h(\gamma) = \gamma$. This proves the claim.

Now by Lemma~\ref{lem:encoding_LP} each $A_i$ and $B_i$ is $O(\ell(p))$--bounded and so each $M(i, j)$ is too. Therefore by Proposition~\ref{prop:small_polytope_vectors}, there is a non-trivial $O(\ell(p))$--bounded vector $v_0$ such that $M(i, j) \cdot v_0 \geq 0$. Then
\[ F_j \cdot 2 v_0 \geq_2 0. \]
Thus there is a multicurve $\gamma \in \calC(S)$ such that $\calT(\gamma) = 2 v_0$. Hence it follows that $\gamma$ is an $h$--invariant multicurve such that $\calT(\gamma)$ is $O(\ell(p))$--bounded.
\end{proof}

As we may choose a path $p$ from $\calT$ to $h(\calT)$ such that $\ell(p) \in O(\ell(h))$ we immediately obtain that:
\begin{corollary}
\label{cor:invariant_bound}
Fix $\calT$, a triangulation of $S$. If $h \in X^*$ is reducible then there is an $h$--invariant multicurve $\gamma \in \calC(S)$ such that $\calT(\gamma)$ is $O(\ell(h))$--bounded. \qed
\end{corollary}

We may use such a multicurve as a certificate that $h \in X^*$ is reducible. Given its edge vector, using Lemma~\ref{lem:is_multicurve_LP} we can first verify that it corresponds to a multicurve $\gamma$ in $O(\ell(h))$ operations. Secondly, by using Lemma~\ref{lem:compute_curve} we can compute $\calT(h(\gamma))$ in $O(\ell(h)^2)$ time. Finally we can verify that $\calT(h(\gamma)) = \calT(\gamma)$, and so verify that $h$ is reducible, in $O(\ell(h))$ time. This shows that:

\begin{corollary}
\label{cor:reducible_NP}
The reducibility problem is in $\NP$. \qed
\end{corollary}

As with all problems in $\NP$, this also gives an exponential time algorithm to decide whether or not a mapping class is reducible. We iterate through the exponentially many $O(\ell(h))$--bounded vectors; if there is a non-trivial one which corresponds to a multicurve and is fixed by $h$ then $h$ is reducible, if not then $h$ is irreducible.

%% file: subsurfaces.tex
\section{Subsurfaces}
\label{sec:subsurfaces}

When $h \in \Mod(S)$ is a reducible mapping class, as well as fixing a multicurve it also fixes a proper subsurface. In order to study the induced mapping class on such an invariant subsurface without talking about surfaces with boundary, we introduce the notion of \emph{crushing} $S$ along a multicurve $\gamma \in \calC(S)$.

\begin{definition}
We \emph{crush} $S$ along $\gamma$ to obtain the (again possibly disconnected) surface $S_\gamma$ by:
\begin{enumerate}
\item removing an open regular neighbourhood of $\gamma$,
\item collapsing the new boundary components to additional marked points, and then
\item removing any components that are twice marked spheres.
\end{enumerate}
See Figure~\ref{fig:crush} for example.
\end{definition}

\begin{figure}[ht]
\centering
\begin{tikzpicture}[scale=2,thick]
\node (a) at (-1.5, 0) {\includegraphics[width=0.5\linewidth, angle=270]{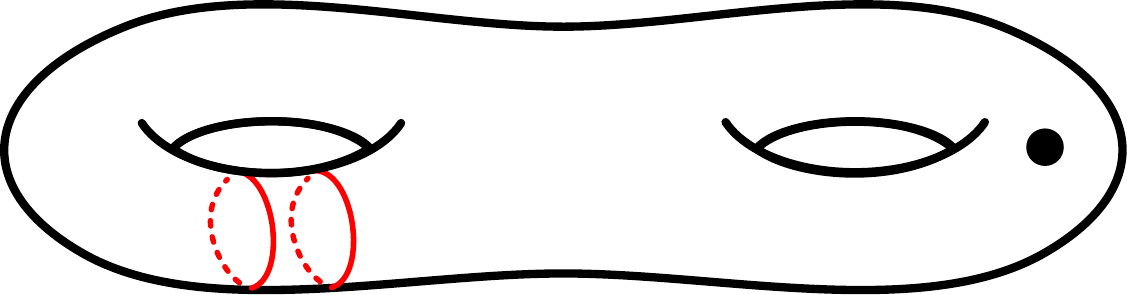}};
\node (b) at ( 1.5, 0) [anchor=east] {\includegraphics[width=0.5\linewidth, angle=270]{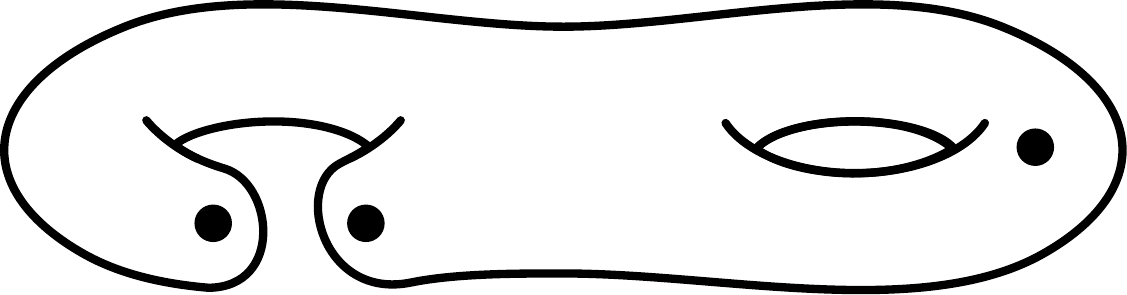}};
\draw [thick,->] ($(a.east)!0.25!(b.west)$) -- node[above] {Crush} ($(a.east)!0.75!(b.west)$);
\end{tikzpicture}
\caption{Crushing along a multicurve.}
\label{fig:crush}
\end{figure}

Now if $\calT$ is a triangulation of $S$ then we may track it as we crush $S$ along a multicurve $\gamma \in \calC(S)$. After possibly collapsing any bigons that are created, this results in a triangulation $\calT_\gamma$ of $S_\gamma$. There is a canonical bijection between the edges of $\calT$ and $\calT_\gamma$ and so $\zeta(S_\gamma) = \zeta$. To see this consider the following construction of $\calT^*_\gamma$, the dual graph of $\calT_\gamma$ inside of $S$:
\begin{enumerate}
\item For each face $f \in F(\calT)$, place a vertex $v$ in the \emph{core} of $f$, that is, the component of $f - \gamma$ which meets all three sides of $f$.
\item Extend three half-edges from $v$ to $\partial f$ whilst avoiding $\gamma$.
\item Extend these half edges along the \emph{corridors} created by parallel strands of $\gamma$ until they connect with another half edge.
\end{enumerate}

\begin{proposition}
\label{prop:crush_paths}
If $p$ is a path from $\calT$ to $\calT'$ then crushing each triangulation of $p$ along $\gamma$, and possibly discarding any repeated triangulations, gives a path $p_\gamma$ from $\calT_\gamma$ to $\calT'_\gamma$ in $G(S_\gamma)$.
\end{proposition}

\begin{proof}
The result clearly holds when $p$ consists of a single reordering of the edges of $\calT$. If $p$ consists of a single flip then the combinatorics of $\calT_\gamma$ and $\calT'_\gamma$ agree away from the faces coming from the faces incident to the flipped edge. Thus $\calT_\gamma$ and $\calT'_\gamma$ share at least $\zeta - 1$ edges and so they are either equal or differ by a single flip. The result then follows for all paths by induction on $\ell(p)$.
\end{proof}

In fact when $\calT'$ is obtained by flipping the edge $e$ of $\calT$, we have that $\calT_\gamma$ and $\calT'_\gamma$ are equal if and only if there is an arc of $\gamma$ passing from one side of the square containing $e$ to the opposite side. Following the notation of Figure~\ref{fig:flip}, this occurs if and only if $\intersection(\gamma, a) + \intersection(\gamma, c) \neq \intersection(\gamma, b) + \intersection(\gamma, d)$.

Finally, we note that by construction $\ell(p_\gamma) \leq \ell(p)$.

\subsection{Maximal curves}

When $\gamma \in \calC(S)$ is an $h$--invariant multicurve, we write $h_\gamma \in \Mod(S_\gamma)$ for the mapping class induced on $S_\gamma$ by $h$. Using this notation, if $p$ is a path from $\calT$ to $h(\calT)$ then $p_\gamma$ is a path from $\calT_\gamma$ to $h_\gamma(\calT_\gamma)$.

\begin{definition}
A multicurve $\gamma \in \calC(S)$ is \emph{$h$--maximal} if it is $h$--invariant and $h_\gamma$ is irreducible.
\end{definition}

Now the bijection between edges of $\calT$ and the edges of $\calT_\gamma$ gives a map $\iota_\gamma \from \calC(S_\gamma) \to \calC(S)$, lifting multicurves on $S_\gamma$ back into $S$. Furthermore, if $\calT(\gamma)$ is $k$--bounded then there is a $k$--bounded integer matrix $M$ such that
\[ \calT(\iota_\gamma(\gamma')) = M \cdot \calT_\gamma(\gamma'). \]
However, it will be easier to work with the map:
\[ \overline{\iota_\gamma} \from \calC(S_\gamma) \to \calC(S) \quad \textrm{given by} \quad \overline{\iota_\gamma}(\gamma') \defeq \iota_\gamma(\gamma') \cup \gamma. \]

\begin{lemma}
\label{lem:curve_lifting_bound}
Suppose that $\calT(\gamma)$ is $k$--bounded. If $\gamma' \in \calC(S_\gamma)$ is multicurve and $\calT_\gamma(\gamma')$ is $k'$--bounded then $\calT(\overline{\iota_\gamma}(\gamma'))$ is $(k + k' + \zeta)$--bounded. \qed
\end{lemma}

We may repeat the construction of an invariant multicurve on $S_\gamma$ and use this bound to control the complexity of the result we obtain back on $S$. To help us do this rigorously we introduce a second notion of complexity, closely related to the dimension of the curve complex of $S_\gamma$ \cite{MasurMinskyII}.

\begin{definition}
The \emph{complexity} of a multicurve $\gamma \in \calC(S)$ is 
\[ \xi(\gamma) \defeq 3\genus(S_\gamma) + \nummarkedpoints(S_\gamma)- 3 |S_\gamma|. \]
\end{definition}

Now note that if $\gamma \in \calC(S)$ and $\gamma' \in \calC(S_\gamma)$ then
\[ \xi(\overline{\iota_\gamma}(\gamma')) < \xi(\gamma). \]
Additionally, $\xi(\gamma) \leq \zeta$ and if $\xi(\gamma) = 0$ then $\calC(S_\gamma) = \emptyset$ and so $\gamma$ must be $h$--maximal.

%%%%%%%%%%%%%%% DONE BELOW %%%%%%%%%%%%%%%%%%%%%%%%%%%%%%%%%55555

\begin{theorem}
\label{thrm:maximal_bound}
Suppose that $h \in \Mod(S)$ is a reducible mapping class and that $p$ is a path from $\calT$ to $h(\calT)$. Then there is an $h$--maximal multicurve $\gamma \in \calC(S)$ such that $\calT(\gamma)$ is $O(\ell(p))$--bounded.
\end{theorem}

\begin{proof}
As $h$ is reducible there is an $h$--invariant multicurve $\gamma \in \calC(S)$ such that $\calT(\gamma)$ is $O(\ell(p))$--bounded by Theorem~\ref{thrm:invariant_bound}.

Now suppose that $\gamma$ is not $h$--maximal. As $h_\gamma$ is reducible, we can reapply Theorem~\ref{thrm:invariant_bound} to the crushed path $p_\gamma$ from $\calT_\gamma$ to $h_\gamma(\calT_\gamma)$. As $\ell(p_\gamma) \leq \ell(p)$, we deduce that there is an $h_\gamma$--invariant multicurve $\gamma' \in \calC(S_\gamma)$ such that $\calT_\gamma(\gamma')$ is $O(\ell(p))$--bounded.

Following this we redefine $\gamma$ to be $\overline{\iota_\gamma}(\gamma')$. This is again an $h$--invariant multicurve and, by Lemma~\ref{lem:curve_lifting_bound}, is still $O(\ell(p))$--bounded. However, doing this decreases $\xi(\gamma)$ and so after repeating this process at most $\zeta$ times $\gamma$ must become $h$--maximal.
\end{proof}

Again, as we may choose a path $p$ from $\calT$ to $h(\calT)$ such that $\ell(p) \in O(\ell(h))$ we immediately obtain that:
\begin{corollary}
\label{cor:maximal_bound}
Fix $\calT$, triangulation of $S$. If $h \in X^*$ is reducible then there is an $h$--maximal multicurve $\gamma \in \calC(S)$ such that $\calT(\gamma)$ is $O(\ell(h))$--bounded. \qed
\end{corollary}

\subsection{The canonical curve system}
\label{sub:canonical}

The \emph{canonical curve system} $\canonical(h) \in \calC(S)$ of a mapping class $h \in \Mod(S)$ is the intersection of all $h$--maximal multicurves \cite[Page 373]{FarbMargalit}. It is non-empty if and only if the mapping class is reducible and of infinite order \cite[Theorem~4.44]{Kida}.

Koberda and Mangahas showed there is an exponential upper bound on the entries of $\calT(\canonical(h))$ \cite[Theorem~1]{KoberdaMangahas}. Corollary~\ref{cor:maximal_bound} also provides an alternate proof of their theorem.

\begin{proposition}
\label{prop:canonical_bound}
Suppose that $h \in \Mod(S)$ is a mapping class and that $p$ is a path from $\calT$ to $h(\calT)$. Then $\calT(\canonical(h))$ is $O(\ell(p))$--bounded.
\end{proposition}

\begin{proof}
If $\canonical(h)$ is empty then the result holds trivially. Otherwise, $h$ is reducible and so by Corollary~\ref{cor:maximal_bound} there is an $h$--maximal multicurve $\gamma \in \calC(S)$ which is $O(\ell(h))$--bounded. Therefore, as $\canonical(h) \subseteq \gamma$, we have that $\canonical(h)$ is $O(\ell(h))$--bounded too.
\end{proof}

Once more, as we may choose a path $p$ from $\calT$ to $h(\calT)$ such that $\ell(p) \in O(\ell(h))$ we immediately obtain that:
\begin{corollary}[{\cite[Theorem~1]{KoberdaMangahas}}]
\label{cor:canonical_bound}
Fix $\calT$, a triangulation of $S$. For each word $h \in X^*$, the edge vector $\calT(\canonical(h))$ of the canonical curve system of $h$ is $O(\ell(h))$--bounded. \qed
\end{corollary}